\documentclass[a4paper,11pt]{amsart}
\usepackage[pdftex]{color,graphicx}
\usepackage{amssymb}
\usepackage{amsfonts}
\usepackage{esint}
\usepackage{amsmath}
\usepackage{amsrefs}
\usepackage{epsfig}
\usepackage{amsthm}
\usepackage{color}
\usepackage[]{epsfig}
\usepackage[]{pstricks}
\usepackage[utf8]{inputenc}

\newtheorem{theorem}{Theorem}[section]
\newtheorem{proposition}{Proposition}[section]
\newtheorem{lemma}{Lemma}[section]
\newtheorem{definition}{Definition}[section]

\newtheorem{corollary}{Corollary}[section]

\newtheorem{remark}{Remark}[section]
\newcommand{\R}{\mathbb{R}}
\newcommand{\h}{\mathbb{H}}
\newcommand{\s}{\mathbb{S}}

\newcommand{\al}{\alpha}
\newcommand{\ria}{\rightarrow}

\newcommand{\om}{\omega}
\newcommand{\n}{\nabla}
\newcommand{\ran}{\rangle}
\newcommand{\lan}{\langle}
\newcommand{\ve}{\varepsilon}

\DeclareMathOperator{\ric}{Ric}

\DeclareMathOperator{\di}{div}
\DeclareMathOperator{\tr}{tr}

\DeclareMathOperator{\diam}{diam}
\DeclareMathOperator{\dist}{dist}
\DeclareMathOperator{\vol}{vol}

\numberwithin{equation}{section}
\title[Poincar\'e and mean value inequalities for hypersurfaces]{Poincar\'e and mean value inequalities for hypersurfaces in Riemannian manifolds and applications}
\author{Hil\'ario Alencar \and Greg\'orio Silva Neto}

\address{Instituto de Matem\'atica,
Universidade Federal de Alagoas,
Macei\'o, AL, 57072-900, Brazil}
\email{hilario@mat.ufal.br}

\address{Instituto de Matem\'atica,
Universidade Federal de Alagoas,
Macei\'o, AL, 57072-900, Brazil}
\email{gregorio@im.ufal.br}

\begin{document}
\footnotetext{H. Alencar was partially supported by CNPq of Brazil}
\begin{abstract}
In the first part of this paper we prove some new Poincar\'e inequalities, with explicit constants, for domains of any hypersurface of a Riemannian manifold with sectional curvatures bounded from above. This inequalities involve the first and the second symmetric functions of the eigenvalues of the second fundamental form of such hypersurface. We apply these inequalities to derive some isoperimetric inequalities and to estimate the volume of domains enclosed by compact self-shrinkers in terms of its scalar curvature. In the second part of the paper we prove some mean value inequalities and as consequences we derive some monotonicity results involving the integral of the mean curvature.
\end{abstract}
\maketitle

\section{Introduction and main results}

The classical Poincar\'e inequality  states that if $\Omega\subset \mathbb{R}^{n}$ is  a bounded, connected, open subset  of $\mathbb{R}^{n}$ and  $1\leq p <\infty,$ then there is a constant $C(p,\Omega)$ depending on $p$ and $\Omega$ such that for all non-negative $f\in W^{1,p}_{0}(\Omega)$ the following inequality holds
\begin{equation}\label{class-poincare}
\left(\int_{\Omega}f^p dx\right)^{\frac{1}{p}}\leq C(p,\Omega)\left(\int_{\Omega}|\n f|^p dx\right)^{\frac{1}{p}},
\end{equation}
where $dx$ denotes the Lebesgue measure of $\R^m.$ If $\Omega=B_r(x_0)$ is the open ball of $\R^m$ with center $x_0\in\R^m$ and radius $r,$ then there is a constant $C(p)$ depending only on $p$ such that 
\[
\int_{B_r(x_0)}f^p dx\leq C(p)r\int_{B_r(x_0)}|\n f|^p dx,
\] 
see \cite[pp. 289-290]{Evans}. The Poincar\'e inequality \eqref{class-poincare} also holds for functions $f\in W^{1,p}(\Omega)$ provided  $\int_{\Omega}fdx=0$, where $\Omega\subset \mathbb{R}^{n}$ is a bounded, connected open subset of $\mathbb{R}^{n}$   with Lipschitz boundary $\partial \Omega$ and for all $1\leq p \leq \infty,$ i.e., there is a constant $C(p,\Omega)$ depending on $p$ and $\Omega$ such that
\begin{equation}\label{class-poincareB}
\left(\int_{\Omega}|f|^{p} dx\right)^{\frac{1}{p}}\leq C(p,\Omega)\left(\int_{\Omega}|\n f|^p dx\right)^{\frac{1}{p}}.
\end{equation} This is the Poincar\'e-Wirtinger inequality.  
An interesting question about these inequalities is to know the  dependence of the Poincar\'{e} constant $C(p,\Omega)$ on the geometry of  the domain $\Omega$ or to find the best constant $C(p,\Omega)$ for a given domain  and a given $p$, $1\leq p\leq\infty$.  For convex domains $\Omega \subset \mathbb{R}^{n}$, Payne-Weinberger showed in \cite{payne-weinberger} that, for  $p=2$, the best Poincar\'{e} constant in inequality \eqref{class-poincareB} is $C(2, \Omega)=\dfrac{1}{\pi}(\diam \Omega),$ and Acosta-Duran showed in \cite[Thm. 3.2]{Acosta} that, for $p=1$, the best Poincar\'e constant is  $C(1, \Omega)=\dfrac{1}{2}(\diam\Omega),$ where here $(\diam\Omega)$ denotes the diameter of $\Omega$ in $\R^n.$ Moreover, they showed that, for any value of $p$,  $1\leq p \leq \infty$, the Poincar\'e constant for convex domains $\Omega \subset \mathbb{R}^{n}$ depends only on the diameter of $\Omega$.
 These Poincar\'e type inequalities can be extended to domains in  Riemannian manifolds. For instance, the following result of P. Li and R. Schoen:
\begin{theorem}[\cite{L-S}, Thm. 1.1, p. 282]
Let $M$ be a compact Riemannian $m$-manifold with boundary $\partial M,$ possibly empty, and with Ricci curvature $\ric_M\geq-(m-1)k$ for a constant $k\geq0.$ Let $x_0\in M$ and $r>0.$ If $\partial M=\emptyset$ assume that the diameter of $M$ is greater than or equal to $2r.$ If $\partial M\neq0,$ assume that the distance from $x_0$ to $\partial M$ is at least $5r.$ For every Lipschitz function $f$ on $B_r(x_0)$ which vanishes on $\partial B_r(x_0)$ we have the Poincar\'e inequality
\[
\int_{B_r(x_0)}|f|dM\leq r(1+\sqrt{k}r)^{-1}e^{2m(1+\sqrt{k}r)}\int_{B_r(x_0)}|\n f|dM.
\]
\end{theorem}
Poincar\'e inequalities on domains of Riemannian manifolds has been studied extensively by many authors and it plays an important role in Geometry and Analysis, see  \cite{Cheng-Zhou}, \cite{Kombe}, \cite{Lam}, \cite{Li}, \cite{Li-Wang} and \cite{Munteanu} for few examples.

Before to state the results of this paper, let us introduce some definitions and notations.  Let $M^m, \ m\geq 2,$ be a $m-$dimensional hypersurface with boundary $\partial M,$ possibly empty, of a Riemannian $(m+1)-$manifold $\overline{M}^{m+1}.$ Let $k_1,k_2,\ldots,k_m$ be the principal curvatures of $M.$ We define the first and the second symmetric functions associated to the principal curvatures of $M$ by 
\begin{equation}\label{defi-S1-S2}
S_1 = \sum_{i=1}^m k_i\ \mbox{e} \ S_2 = \sum_{i<j}^m k_i k_j.
\end{equation}
These functions have natural geometric meaning. In fact, $S_1=mH,$ where $H$ denotes the mean curvature of $M.$ If $\overline{M}^{m+1}$ has constant sectional curvature $\kappa$, then $2S_2=m(m-1)(R-\kappa),$ where $R$ is the scalar curvature of $M.$

If $K_{\overline{M}}=K_{\overline{M}}(x,\Pi_x)$ denotes the sectional curvature of $\overline{M}$ in $x\in M$ relative to the $2-$ dimensional subspace $\Pi_x\subset T_x\overline{M},$ we define
\[
\displaystyle{\kappa_0(x)=\inf_{\Pi_x\subset T_x\overline{M}}K_{\overline{M}}(x,\Pi_x).}
\]

Let $i(\overline{M})$ be the injectivity radius of $\overline{M}$ and let us denote by $(\diam\Omega)$ the diameter of smallest extrinsic ball which contains $\Omega\subset M.$ 


In this paper we will address Poincar\'e type inequalities, with explicit constants, on domains of any hypersurface of a Riemannian manifold with sectional curvatures bounded from above. 

\begin{theorem}[Poincar\'e type inequality]\label{Theo-Poincare}
Let $\overline{M}^{m+1}$ be a Riemannian $(m+1)-$manifold, $m\geq 2,$ with sectional curvatures bounded from above by a constant $\kappa.$ Let $M$ be a hypersurface of $\overline{M}^{m+1},$ with boundary $\partial M,$ possibly empty, such that $S_1>0$ and $S_2\geq0.$ Let $\Omega\subset M$ be a connected and open domain with compact closure. If $\partial M\neq\emptyset,$ assume in addition that $\overline{\Omega}\cap \partial M=\emptyset.$ If $(\diam\Omega)<2i(\overline{M})$ then, for every non-negative $\mathcal{C}^{1}$-function $f\colon M\ria\R,$  compactly supported in $\Omega,$ we have
\begin{equation}\label{Poincare-ineq}
\int_\Omega fS_1dM \leq C(\Omega)\int_\Omega\left[|\n f|S_1 + \left(\frac{m(\kappa-\kappa_0)}{4} + S_2\right)f\right]dM, 
\end{equation}
where
\[
C(\Omega) = C(1,\Omega)=\left\{
\begin{array}{cc}
\dfrac{1}{m-1}(\diam\Omega)&\ \mbox{if} \ \kappa\leq 0;\\
\dfrac{2}{\sqrt{\kappa}(m-1)}\tan\left(\frac{\sqrt{\kappa}}{2}(\diam\Omega)\right), &\ \mbox{if} \ \kappa>0;
\end{array}
\right.
\]
and we assume $(\diam \Omega) <\frac{\pi}{\sqrt{\kappa}}$ for $\kappa>0.$ Moreover, if $M=\s^m(r)$ is the Euclidean sphere of radius $r,$ $\overline{M}=\R^{m+1},$ and $f$ is a constant function, then the equality holds.
\end{theorem}



\begin{remark}\label{rem.inj}
{\normalfont 
If $\kappa\leq0$ and $\overline{M}$ is complete and simply connected, then $i(\overline{M})=\infty.$ If $\kappa>0$ and the sectional curvatures of $\overline{M}$ are pinched between $\frac{1}{4}\kappa$ and $\kappa,$ then $i(\overline{M})\geq\frac{\pi}{\sqrt{\kappa}},$ see for example \cite{dC}, p.$276.$ For these situations the assumption $(\diam\Omega)\leq 2i(\overline{M})$ in Theorem \ref{Theo-Poincare} is automatically satisfied. 
}
\end{remark}

\begin{remark}\label{rem.seo}
{\normalfont
Results in the direction of the Theorem \ref{Theo-Poincare} for mean curvature are known, see \cite[Theorem 3.1 and Theorem 3.3, p. 531-532]{Seo}.
}
\end{remark}

In the particular case when $\overline{M}^{m+1}$ is one of the space forms $\R^{m+1},$ $\h^{m+1}(\kappa),$ or $\s^{m+1}(\kappa),$ namely, the Euclidean space, the hyperbolic space of curvature $\kappa<0,$ and the Euclidean sphere of curvature $\kappa>0,$ respectively, we have

\begin{corollary}
Let $M$ be a hypersurface with boundary $\partial M,$ possibly empty, of $\R^{m+1},$ $\h^{m+1}(\kappa)$ or $\s^{m+1}(\kappa),$  $m\geq 2,$ with mean curvature $H>0$ and scalar curvature $R\geq \kappa.$ Let $\Omega\subset M$ be a connected and open domain with compact closure. If $\partial M\neq\emptyset,$ assume in addition that $\overline{\Omega}\cap \partial M=\emptyset.$ Then, for every non-negative $\mathcal{C}^{1}$-function $f\colon M\ria\R,$ compactly supported in $\Omega,$ we have
\begin{equation}
\int_\Omega fHdM \leq C(\Omega)\int_\Omega\left[|\n f|H + \frac{(m-1)}{2}(R-\kappa)f\right]dM.
\end{equation}
Moreover, if $M=\s^m(r)$ is the Euclidean sphere of radius $r$ in $\R^{m+1}$ and $f$ is a constant function, then the equality holds.
\end{corollary}


Theorem \ref{Theo-Poincare} has various immediate consequences that we will present below. The next result should be compared with \cite[Theorem 28.2.5, p. 210]{burago} which states that

\emph{
If $M^m$ is a compact submanifold with angles of $\R^n,\ n>m,$ with boundary $\partial M,$ possibly empty, then
\[
m \vol(M) \leq R_M \left[\vol(\partial M)+m\int_M|H|dM\right],
\]
where $R_M$ is the radius of the smallest ball of $\R^n$ which contains $M.$
}

As a consequence of Theorem \ref{Theo-Poincare} we have 


\begin{corollary}\label{isoperimetric}
Let $\overline{M}^{m+1}$ be a Riemannian $(m+1)-$manifold, $m\geq 2,$ with sectional curvatures bounded from above by a constant $\kappa.$ Let $M$ be a hypersurface of $\overline{M}^{m+1},$ with boundary $\partial M,$ possibly empty, such that $S_1>0$ and $S_2\geq0.$ Let $\Omega\subset M$ be a connected and open domain with compact closure. If $\partial M\neq\emptyset,$ assume in addition that $\overline{\Omega}\cap \partial M=\emptyset.$ If $(\diam\Omega)<2i(\overline{M})$ and, for $\kappa>0,$ assuming also $(\diam\Omega)<\frac{\pi}{\sqrt{\kappa}},$ then
\[
\int_\Omega S_1dM \leq C(\Omega)\left[\int_{\partial\Omega}S_1dS_\Omega +\int_\Omega \left(\frac{m(\kappa-\kappa_0)}{4} + S_2 \right)dM\right],
\]
where $dS_\Omega$ denotes the $(m-1)-$dimensional measure of $\partial\Omega.$ Moreover, if $M$ is compact without boundary, then
\[
\int_M S_1dM \leq C(\Omega)\int_M \left(\frac{m(\kappa-\kappa_0)}{4}+S_2\right)dM.
\]
In particular, if $\overline{M}$ has constant sectional curvature $\kappa,$ then
\[
\int_\Omega  H dM \leq C(\Omega)\left[\int_{\partial\Omega} H dS_\Omega +\frac{m-1}{2}\int_\Omega (R-\kappa)dM\right].
\]

\end{corollary}

If $M$ is a compact, without boundary, hypersurface of $\R^{m+1},$ $\h^{m+1}(\kappa)$ or $\s^{m+1}(\kappa),$ the Corollary \ref{isoperimetric} becomes

\begin{corollary}\label{iso-2}
Let $M$ be a compact, without boundary, hypersurface of $\R^{m+1},\ \h^{m+1}(\kappa)$ and $\s^{m+1}(\kappa)$ with mean curvature $H>0$ and scalar curvature $R\geq \kappa.$ 
\begin{itemize}
\item[(i)] If $M\subset\R^{m+1}$ or $M\subset\h^{m+1}(\kappa),$ then 
\[
\int_M H dM \leq \frac{1}{2}(\diam M)\int_M (R-\kappa)dM
\]
and
\[
\diam M\geq \dfrac{2\min H}{\max R - \kappa}.
\]
In particular, if $m=2$ and $M^2\subset\R^3,$ then
\[
\int_M H dM\leq 2\pi(\diam M).
\]
\item[(ii)] If $M\subset \s^{m+1}(\kappa)$ and $(\diam M)\leq \frac{\pi}{\sqrt{\kappa}},$ then
\[
\int_M H dM \leq \frac{1}{\sqrt{\kappa}}\tan\left(\frac{\sqrt{\kappa}}{2}(\diam M)\right)\int_M (R-\kappa)dM.
\]
\end{itemize}
\end{corollary}

\begin{remark}
{\normalfont
If $M=\s^{m}(r),$ the round spheres of radius $r$ in the Euclidean space $\R^{m+1},$ then the inequalities of the item (i) of Corollary \ref{iso-2} become equality.
}
\end{remark}
\begin{remark}
{\normalfont
The proof of the item (i) of the Corollary \ref{isoperimetric} follows taking the compactly supported Lipschitz function $f_\ve:M\ria\R,$ given by
\[
f_\ve(x)=\left\{
\begin{array}{cll}
1&\mbox{if}& x\in\Omega, \dist(x,\partial\Omega)\geq \ve;\\
\dfrac{\dist(x,\partial\Omega)}{\ve}& \mbox{if}& x\in\Omega, \dist(x,\partial\Omega)\leq \ve;\\
0& \mbox{if}& x\not\in\Omega;
\end{array}
\right.
\]
into the Poincar\'e inequality (\ref{Poincare-ineq}) of Theorem \ref{Theo-Poincare} and by taking $\ve\ria0,$ where $\dist(x,\partial\Omega)$ denotes the intrinsic distance of $M$ of the point $x\in M$ to the boundary $\partial M$ of $M.$ 
}
\end{remark}

A hypersurface $M\subset\R^{m+1}$ is called a self-shrinker if its mean curvature $H$ satisfies $H=-\frac{1}{2m}\lan \overline{X},\eta\ran,$ where $\overline{X}$ is the position vector of $\R^{m+1}$ and $\eta$ is the normal vector field of $M$. Self-shrinkers play an important role in the study of mean curvature flow and they have been extensively studied in recent years, see  \cite{cao-zhou}, \cite{cold-mini-1}, and \cite{cold-mini-2}. Another consequence of our version of  Poincar\'e inequality is the following result.

\begin{corollary}\label{self-shrinker}
Let $M$ be a compact, without boundary, self-shrinker of $\R^{m+1},$ $m\geq 2,$ with mean curvature $H>0$ and scalar curvature $R\geq0.$ Let $K\subset\R^{m+1}$ be the compact domain such that $M=\partial K.$ Then
\begin{equation}\label{self}
\vol(K)\leq \frac{m}{m+1}(\diam M)\int_M R dM.
\end{equation}
Moreover, if $M=\s^m(\sqrt{2m})$ is the Euclidean sphere of radius $\sqrt{2m},$ then the equality holds.
\end{corollary}
The proof follows observing that using divergence theorem,
\[
\int_M H dM =\int_{\partial K}\frac{1}{2m}\lan\overline{X},-\eta\ran dM = \frac{1}{2m}\int_K \di_{\R^{m+1}}\overline{X}dV = \frac{m+1}{2m}\vol(K)
\]
and then replacing the identity above into the item (ii) of Corollary \ref{isoperimetric}, for $\kappa=\kappa_0=0$ and $2S_2=m(m-1)R$. The Corollary \ref{self-shrinker} can be compared with Corollary 1.2, p. 4718 of \cite{Guan-Li}, which proves an estimate of the volume of a convex domain in the Euclidean space by the integral of the curvatures of its boundary.

Let $\overline{M}^{m+1}$ be a Hadamard manifold, i.e., a complete and simply connected Riemannian manifold with non-positive sectional curvatures. By the Remark \ref{rem.inj}, $i(\overline{M})=\infty.$ If $M$ is a compact hypersurface of $\overline{M}^{m+1}$, without boundary, then using the well known D. Hoffman and J. Spruck's Sobolev inequality, see \cite[Thm. 2.1, p. 716]{H-S}, we obtain an estimate for the volume of $M$. 

\begin{corollary}
Let $\overline{M}^{m+1}, m\geq 2,$ be a Hadamard manifold with sectional curvatures bounded from above by a constant $\kappa\leq 0.$ Let $M$ be a compact hypersurface of $\overline{M}^{m+1},$ without boundary, such that $S_1>0$ and $S_2\geq 0.$ Then
\begin{equation}\label{vol.est}
\vol(M)^{\frac{m-1}{m}}\leq \frac{2^{m-1}(m+1)^{1+1/m}}{m(m-1)^2\om_m^{1/m}}(\diam M)\int_M \left(\dfrac{m(\kappa-\kappa_0)}{4}+S_2\right)dM,
\end{equation}
where $\om_m$ denotes the volume of the Euclidean unit sphere.
\end{corollary} 
In fact, by using the Hoffman-Spruck-Sobolev inequality
\[
\left[\int_M f^{\frac{m}{m-1}}dM\right]^{\frac{m}{m-1}}\leq\frac{2^{m-2}(m+1)^{1+1/m}}{(m-1)\om_m^{1/m}}\int_M[|\n f| + f|H|]dM
\]
for $f\equiv 1,$ in addition with the Corollary \ref{isoperimetric}, we have
\[
\begin{split}
\vol(M)^{\frac{m}{m-1}}&\leq\frac{2^{m-2}(m+1)^{1+1/m}}{m(m-1)\om_m^{1/m}}\int_M S_1dM\\
&\leq \frac{2^{m-1}(m+1)^{1+1/m}}{m(m-1)^2\om_m^{1/m}}(\diam M)\int_M \left(\dfrac{m(\kappa-\kappa_0)}{4}+S_2\right)dM.\\
\end{split}
\]

\subsection{Mean value inequalities} The techniques used in the proof of the Poincar\'e inequalities can be extended to prove some mean value inequalities involving both $S_1$ and $S_2.$ Results in this direction are known for hypersurfaces of the Euclidean space. In this case, holds 

\begin{theorem}[\cite{colding-minicozzi}, Lemma 1.18] Let $M\subset\R^{m+1}$ be a properly immersed hypersurface with mean curvature $H$ and $f$ is a non-negative function on $M,$ then for $s<t,$
\begin{equation}\label{eqCM1}
\frac{1}{t^m}\int_{M\cap B_t} f dM - \frac{1}{s^m}\int_{M\cap B_s} f dM \geq \int_s^t\frac{1}{r^{m+1}}\int_{M\cap B_r}\lan \overline{X},\n f + fH\eta\ran dM dr,
\end{equation}
where $\overline{X}$ is the position vector of $\R^{m+1}$ and $\eta$ is the normal vector field of $M.$\end{theorem}

In the same spirit of inequality (\ref{eqCM1}), we obtain

\begin{theorem}[Mean Value Inequalities]\label{Theo-Mean} Let $\overline{M}^{m+1},\ m\geq2,$ be a Riemannian $(m+1)-$manifold with sectional curvatures bounded from above by a constant $\kappa.$ Let $M$ be a proper hypersurface of $\overline{M}$ with $S_1>0$ and $S_2\geq 0.$ Let $x_0$ be a point of $\overline{M}^{m+1},$ $\rho(x)=\rho(x_0,x)$ be the extrinsic distance function starting at $x_0$ to $x\in M,$ and $B_r=B_r(x_0)$ be the extrinsic ball of center $x_0$ and radius $r.$ If $\partial M\neq 0,$ assume that $B_r\cap\partial M=\emptyset.$ Then, for any non-negative, locally integrable, $\mathcal{C}^1$-function $f:M\ria\R$ and for any $0<s<t<i(\overline{M}),$ we have
\begin{itemize}
\item[(i)] for $\kappa\leq 0,$
\[
\begin{split}
\dfrac{1}{t^{\frac{m-1}{2}}}& \int_{M\cap B_t}fS_1dM - \dfrac{1}{s^{\frac{m-1}{2}}}\int_{M\cap B_s}fS_1dM\\
&\geq \dfrac{1}{2}\int_s^t \dfrac{1}{r^{\frac{m+1}{2}}}\int_{M\cap B_r}\rho\left[\left\lan\overline{\n}\rho,\left(S_1I-A\right)(\n f) +2S_2f\eta\right\ran + f\ric_{\overline{M}}(\n\rho,\eta)\right]dM dr;
\end{split}
\]
\item[(ii)] for $\kappa> 0,$
\[
\begin{split}
&\dfrac{1}{(\sin\sqrt{\kappa} t)^{\frac{m-1}{2}}} \int_{M\cap B_t}fS_1dM - \dfrac{1}{(\sin\sqrt{\kappa} s)^{\frac{m-1}{2}}}\int_{M\cap B_s}fS_1dM\\
&\qquad\geq \dfrac{1}{2}\int_s^t\!\!\!\!\frac{1}{(\sin\sqrt{\kappa}r)^{\frac{m+1}{2}}}\int_{M\cap B_r}\!\!\!\!\!\!(\sin\sqrt{\kappa}\rho)\left[\left\lan\!\overline{\n}\rho,\left(S_1I-A\right)\!(\n f) +2S_2f\eta\right\ran + f\ric_{\overline{M}}(\n\rho,\eta)\right]dM dr,
\end{split}
\]
provided $t<\dfrac{\pi}{2\sqrt{\kappa}}.$
\end{itemize}

Here $\ric_{\overline{M}}$ denotes the Ricci tensor of $\overline{M},$ $A:TM\ria TM$ denotes the linear operator associated with the second fundamental form of $M$, defined in the tangent bundle $TM$ of $M,$ and $I:TM\ria TM$ denotes the identity operator.
\end{theorem}
We apply the mean value inequalities of the Theorem \ref{Theo-Mean} to obtain some monotonicity results involving the integral of the mean curvature. Monotonicity results appear in several branches of Analysis and Riemannian Geometry in the study to determine the variational behaviour of geometric quantities, see for example \cite{walcy}, \cite{colding-minicozzi}, \cite{colding}, \cite{cold-mini}, \cite{Ecker1}, \cite{Ecker2}, \cite{Gruter}, \cite{Li-JF}, \cite{Simon}, and \cite{Urbas}.


\begin{corollary}[Monotonicity]\label{monotonicity} Let $\overline{M}^{m+1},\ m\geq2,$ be a Riemannian $(m+1)-$manifold with sectional curvatures bounded from above by a constant $\kappa.$ Let $M$ be a proper hypersurface of $\overline{M}$ with $S_1>0$ and $S_2\geq 0.$ If there exists $0<\alpha\leq1, \ \Lambda\geq0$ and $0<R_0<i(\overline{M})$ such that
\begin{equation}\label{hyp-mono-geral}
\al^{-1}\int_{M\cap B_r}\left(\dfrac{m(\kappa-\kappa_0)}{4} + S_2\right)dM\leq\Lambda\left(\dfrac{r}{R_0}\right)^{\al-1}\int_{M\cap B_r} S_1 dM,
\end{equation}
for all $r\in(0,R_0),$ then
\begin{itemize}
\item[(i)] for $\kappa\leq0,$ the function $h:(0,R_0)\ria\R$ defined by
\[
h(r)=\dfrac{\exp(\Lambda R_0^{1-\al}r^\al)}{r^{\frac{m-1}{2}}}\int_{M\cap B_r} S_1 dM
\]
is monotone non-decreasing;

\item[(ii)] for $\kappa>0$ and $\kappa R_0^2\leq\pi^2,$ the function $h:(0,R_0)\ria\R$ defined by
\[
h(r)=\dfrac{\exp(\Lambda R_0^{1-\al}r^\al)}{(\sin\sqrt{\kappa}r)^{\frac{m-1}{2}}}\int_{M\cap B_r} S_1dM
\]
is monotone non-decreasing.
\end{itemize}
\end{corollary}

In particular, if $\overline{M}=\R^{m+1}$ or $\s^{m+1}(\kappa)$ we have, taking $\alpha=1$ in the Corollary \ref{monotonicity}, the following result.
\begin{corollary}[Monotonicity] Let $M$ be a proper hypersurface of $\R^{m+1}$ or $\s^{m+1}(\kappa), \ m\geq2,$ with mean curvature $H>0$ and scalar curvature $R\geq\kappa.$ If there exists $\Lambda\geq0$ such that 
\begin{equation}\label{hyp.mono}
\kappa\leq R\leq \frac{\Lambda}{2}H+\kappa,
\end{equation}
then
\begin{itemize}
\item[(i)] for $M^m\subset\R^{m+1},$ the function $\varphi:(0,\infty)\subset\R\ria\R,$ defined by
\[
\varphi(r)=\dfrac{e^{\frac{\Lambda}{2}r}}{r^{\frac{m-1}{2}}}\int_{M\cap B_r}HdM
\]
is monotone non-decreasing;
\item[(ii)] for $M^m\subset\s^{m+1}(\kappa),$ the function $\varphi:\left(0,\frac{\pi}{2\sqrt{\kappa}}\right)\ria\R$ defined by
\[
\varphi(r)=\dfrac{e^{\frac{\Lambda}{2}r}}{(\sin\sqrt{\kappa}r)^{\frac{m-1}{2}}}\int_{M\cap B_r}HdM
\]
is monotone non-decreasing.
\end{itemize}

\end{corollary}
\begin{remark}
{\normalfont 
For the case when $\overline{M}=\h^{m+1}(\kappa),$ results were obtained in \cite{A-S-arkiv}.
}
\end{remark}

In the next application of the mean value inequalities we study the behaviour of the integral of the mean curvature when we assume $L^p$ bounds for the scalar curvature.

\begin{corollary}\label{Lp}
Let $\overline{M}^{m+1},\ m\geq2,$ be a Riemannian $(m+1)-$manifold with sectional curvatures bounded from above by a constant $\kappa.$ Let $M$ be a proper hypersurface of $\overline{M}$ with $S_1\geq c$ for some constant $c>0$ and $S_2\geq 0.$ If there exists $0<R_0<i(\overline{M}), \Lambda>0$ and $p>1$ such that
\begin{equation}
\left[\int_{M\cap B_{R_0}} \left(\dfrac{m(\kappa-\kappa_0)}{4} + S_2 \right)^p dM\right]^{1/p}\leq\Lambda,
\end{equation}
then for every $0<s<t<R_0,$ we have
\begin{itemize}
\item[(i)] for $\kappa\leq 0,$
\[
\left(\dfrac{1}{s^{\frac{m-1}{2}}}\int_{M\cap B_s}S_1dM\right)^{1/p}\leq\left(\dfrac{1}{t^{\frac{m-1}{2}}}\int_{M\cap B_t}S_1dM\right)^{1/p} + \frac{2\Lambda}{c^{1-1/p}(m-1-2p)}\int_s^t\frac{1}{r^{\frac{m-1}{2p}}}dr;
\]
\item[(ii)] for $\kappa>0$ and $R_0\leq\dfrac{\pi}{2\sqrt{\kappa}},$
\[
\begin{split}
\left(\dfrac{1}{(\sin \sqrt{\kappa}s)^{\frac{m-1}{2}}}\int_{M\cap B_s}S_1dM\right)^{1/p}&\leq\left(\dfrac{1}{(\sin \sqrt{\kappa}t)^{\frac{m-1}{2}}}\int_{M\cap B_t}S_1dM\right)^{1/p}\\
&\qquad\qquad + \frac{\Lambda(m-1)}{c^{1-1/p}(m-1-2p)}\int_s^t\frac{1}{(\sin \sqrt{\kappa}r)^{\frac{m-1}{2p}}}dr.
\end{split}
\]
\end{itemize}
\end{corollary}


We conclude the applications of the mean value inequalities with the following monotonicity result for self-shrinkers.

\begin{corollary}\label{teo-self}
Let $M$ be a proper self-shrinker of $\R^{m+1},$ $m\geq2,$ with mean curvature $H>0$ and scalar curvature $0 \leq R \leq \Lambda$ for some constant $\Lambda\geq0.$ Then the function $\varphi:(0,\infty)\ria\R$ defined by
\[
\varphi(r)=\frac{1}{r^{(m-1)\left(\frac{1}{2} - m\Lambda\right)}}\int_{M\cap B_r} H dM
\]
is monotone non-decreasing. Moreover, if $M$ is complete, non-compact and the scalar curvature satisfies $0\leq R \leq \Lambda <\frac{1}{2m},$ then $\displaystyle{\int_M H dM =\infty.}$
\end{corollary}


This paper is organized in four sections as follows. In the section 2 we present some preliminary results which give basis to establish the argument used in this paper, including Proposition \ref{prop-div-P1-const} and Lemma \ref{lemma-main1}. In section 3, we prove the Poincar\'e type inequalities and, in section 4, we prove the mean value inequalities and its consequences. 

{\bf Acknowledgements:} The authors are grateful to Detang Zhou and Greg\'orio Pacelli Bessa for their suggestions.

\section{Preliminary results}
Let $M$ be a $m-$dimensional hypersurface of a Riemannian $(m+1)-$manifold $\overline{M}, \ m\geq2.$ Denote by $\n$ and $\overline{\n}$ the connections of $M$ and $\overline{M},$ respectively. Let $\overline{X}:M\ria T\overline{M}$ be a vector field and write $\overline{X}=X^T + X^N,$ where $X^T\in TM$ and $X^N\in TM^\perp.$ Let $Y,Z\in TM$ be vector fields, and denote by $\lan\cdot,\cdot \ran$ the metric of $\overline{M}.$ Then
\[
\begin{split}
\lan \overline{\n}_Y\overline{X},Z\ran &= \lan \overline{\n}_Y X^T + \overline{\n}_Y X^N,Z\ran\\
&=\lan \overline{\n}_Y X^T,Z\ran + \lan\overline{\n}_Y X^N,Z\ran\\
&=\lan \overline{\n}_Y X^T,Z\ran - \lan X^N,\overline{\n}_YZ\ran\\
&=\lan \overline{\n}_Y X^T,Z\ran - \lan X^N,B(Y,Z)\ran,\\
\end{split}
\]
where $B(Y,Z)=\overline{\n}_YZ - \n_YZ$ denotes the bilinear form associated with the second fundamental form of $M$. If $\eta$ denotes the unit normal vector field of $M$, then $X^N=\lan \overline{X},\eta\ran\eta.$ It implies
\begin{equation}\label{sec.fund.form}
\begin{split}
\lan \overline{\n}_Y\overline{X},Z\ran &= \lan \overline{\n}_Y X^T,Z\ran - \lan\overline{X},\eta\ran \lan \eta,B(Y,Z)\ran\\
&=\lan \overline{\n}_Y X^T,Z\ran - \lan\overline{X},\eta\ran\lan A(Y),Z\ran,
\end{split}
\end{equation}
where $A:TM\ria TM$ is the linear operator defined by
\begin{equation}\label{shape}
\lan A(V),W\ran = \lan \eta, B(V,W)\ran, \ V,W\in TM.
\end{equation}
\begin{definition}
Let $A:TM\ria TM$ be defined by (\ref{shape}) the linear operator associated to the second fundamental form of $M$. The first Newton transformation $P_1:TM\ria TM$ is defined by
\[
P_1=S_1I-A,
\]
where $S_1=\tr_M A,$ $\tr_M$ denotes the trace in $M,$ and $I:TM\ria TM$ is the identity map.
\end{definition}
\begin{remark}\label{P1-selfadjoint}
{\normalfont
Notice that, since $A$ is self-adjoint, then $P_1$ is also a self-adjoint linear operator. Denote by $k_1,k_2,\ldots, k_m$ the eigenvalues of the linear operator $A,$ also called principal curvatures of $M$. Since $P_1$ is a self-adjoint operator, we can consider its eigenvalues $\theta_1,\theta_2,\ldots,\theta_m$ given by $\theta_i = S_1 - k_i,$ $i=1,2,\ldots,m.$
}
\end{remark}
\begin{remark}\label{P1-positive}
{\normalfont
If $S_1>0$ and $S_2\geq0,$ then $P_1$ is positive semidefinite. This fact is known, and can be found in \cite[Remark 2.1, p. 552]{AdCS}, however, we present a proof here for the sake of completeness. If $S_2\geq0,$ then $S_1^2 = |A|^2 + 2S_2 \geq k_i^2, \ \mbox{for all}\ i=1,2,\ldots,m.$ Thus, $0\leq S_1^2 - k_i^2 = (S_1 - k_i)(S_1 + k_i)$ which implies that all eigenvalues of $P_1$ are non-negative, provided $S_1>0,$ i.e., $P_1$ is positive semidefinite.
}
\end{remark}

The following result is known and we include a proof here for the sake of completeness.

\begin{lemma}\label{lemma-divP1}
If $(\di_M P_1)(V)=\tr_M(E\mapsto(\n_E P_1)(V)),$ where $(\n_E P_1)(V)=\n_E(P_1(V))-P_1(\n_EV),$ then 
\[(\di_M P_1)(V)=\ric_{\overline{M}}(V,\eta)\]
for every $V\in TM,$ where $\eta$ denotes the unitary normal vector field of $M$ and $\ric_{\overline{M}}$ denotes the Ricci tensor of $\overline{M}.$ In particular, if $\overline{M}^{m+1}$ has constant sectional curvature or is an Einstein manifold, then $\di_M P_1 =0.$
\end{lemma}
\begin{proof}
Let $\{e_1,e_2,\ldots,e_m\}$ be an orthonormal referential in $TM$ which is geodesic at $p\in M.$ By using the Codazzi equation
\[
\lan(\n_V A)(Y),Z\ran - \lan(\n_Y A)(V),Z\ran = \lan\overline{R}(Y,V)Z,\eta\ran
\]
for $Y=Z=e_i$ and summing over $i$ from $1$ to $m,$ we have
\[
\sum_{i=1}^m \lan(\n_V A)(e_i),e_i\ran - \lan(\n_{e_i} A)(V),e_i\ran = \lan\overline{R}(e_i,V)e_i,\eta\ran,
\]
i.e.,
\[
\sum_{i=1}^m \lan(\n_V A)(e_i),e_i\ran - (\di_M A)(V) = \ric_{\overline{M}}(V,\eta).
\]
Observing that
\[
\sum_{i=1}^m \lan(\n_V A)(e_i),e_i\ran = \sum_{i=1}^m V\lan A(e_i),e_i\ran = V(S_1) = \di_M(S_1 I)(V),
\]
where $I:TM\ria TM$ is the identity operator, we conclude that
\[
\di_M(S_1 I - A)(V) = \ric_{\overline{M}}(V,\eta).
\]
In particular, if $\overline{M}$ has constant sectional curvature $\kappa,$ then $\ric_{\overline{M}}(V,\eta)=m\kappa\lan V,\eta\ran=0,$ which implies $\di_M P_1=0.$ Also, if $\overline{M}$ is an Einstein manifold with Einstein constant $\lambda,$ we have $\ric_{\overline{M}}(V,\eta)=\lambda\lan V,\eta\ran=0,.$ Therefore $(\di_M P_1)=0.$
\end{proof}

The next result is an important tool in the proof of the Theorem \ref{Theo-Poincare}.

\begin{proposition}\label{prop-div-P1-const} If $M^m, m\geq2,$ is an hypersurface of a Riemannian $(m+1)-$manifold $\overline{M}^{m+1}$ and $\overline{X}:M\ria T\overline{M}$ is a vector field, then
\begin{equation}\label{div-Pr-const}
\di_M(P_1(X^T)) = \tr_M\left(E\longmapsto P_1\left(\left(\overline{\n}_E\overline{X}\right)^T\right)\right) + \ric_{\overline{M}}(X^T,\eta) + 2S_2\lan \overline{X},\eta\ran,
\end{equation}
where $\ric_{\overline{M}}$ denotes the Ricci tensor of $\overline{M}$ and $X^T=\overline{X}-\lan \overline{X},\eta\ran\eta$ is the tangent part of $\overline{X}.$
\end{proposition}
\begin{proof}
Let $\{e_1,e_2,\ldots,e_m\}$ be an orthonormal referential in $TM.$ First, since $P_1$ is self-adjoint, we have 
\begin{equation}\label{eq.trXP_1}
\begin{split}
\tr_M\left(E\longmapsto P_1\left(\left(\overline{\n}_E\overline{X}\right)^T\right)\right) &= \sum_{i=1}^m \left\lan P_1\left(\left(\overline{\n}_{e_i}\overline{X}\right)^T\right),e_i\right\ran =\sum_{i=1}^m \lan\overline{\n}_{e_i}\overline{X},P_1(e_i)\ran.\\
\end{split}
\end{equation}
By using (\ref{sec.fund.form}), p.\pageref{sec.fund.form}, and the self-adjointness of $A,$ we obtain
\[
\begin{split}
\sum_{i=1}^m \lan \overline{\n}_{e_i}\overline{X}, P_1(e_i)\ran &= \sum_{i=1}^m \lan \overline{\n}_{e_i} X^T, P_1(e_i)\ran - \left(\sum_{i=1}^m \lan A(e_i),P_1(e_i)\ran\right)\lan \overline{X},\eta\ran\\
&=\sum_{i=1}^m \lan \overline{\n}_{e_i} X^T, P_1(e_i)\ran - \left(\sum_{i=1}^m \lan (A\circ P_1)(e_i),e_i\ran\right)\lan \overline{X},\eta\ran\\
&=\sum_{i=1}^m \lan\overline{\n}_{e_i}X^T, P_1(e_i)\ran - \tr_M(A\circ P_1)\lan\overline{X},\eta\ran.
\end{split}
\]
Thus,
\[
\sum_{i=1}^m \lan\overline{\n}_{e_i}X^T, P_1(e_i)\ran = \tr_M\left(E\longmapsto P_1\left(\left(\overline{\n}_E\overline{X}\right)^T\right)\right) + \tr_M(A\circ P_1)\lan \overline{X},\eta\ran.
\]
On the other hand, the self-adjointness of $P_1$ implies
\[
\begin{split}
\sum_{i=1}^m \lan \overline{\n}_{e_i} X^T,P_1(e_i)\ran &= \sum_{i=1}^m \lan \n_{e_i} X^T + B(e_i,X^T),P_1(e_i)\ran = \sum_{i=1}^m \lan \n_{e_i} X^T ,P_1(e_i)\ran\\
&=\sum_{i=1}^m \lan P_1(\n_{e_i} X^T), e_i\ran = \sum_{i=1}^m \lan \n_{e_i}(P_1(X^T)),e_i\ran - \sum_{i=1}^m \lan (\n_{e_i} P_1)(X^T),e_i\ran\\
&=\di_M(P_1(X^T)) - \tr_M(E\ria (\n_E P_1)(X^T))\\
&=\di_M(P_1(X^T)) - (\di_M P_1)(X^T).
\end{split}
\]
Therefore, 
\[
\di_M(P_1(X^T)) = \tr_M\left(E\longmapsto P_1\left(\left(\overline{\n}_E\overline{X}\right)^T\right)\right) + (\di_M P_1)(X^T) + \tr_M(A\circ P_1)\lan \overline{X},\eta\ran.
\]
The result follows using the Lemma \ref{lemma-divP1} and the fact
\[
\tr_M(A\circ P_1) = \tr_M(A\circ (S_1I - A)) = S_1\tr_M(A) - \tr_M(A^2) = S_1^2 - |A|^2 = 2S_2.
\]
\end{proof}

\begin{remark}\label{remark-ricci}
{\normalfont
If the sectional curvatures $K_{\overline{M}}$ of $\overline{M}^{m+1}$ satisfy
\[
\kappa_0\leq K_{\overline{M}}\leq \kappa
\]
for real numbers $\kappa_0$ e $\kappa,$ then
\begin{equation}\label{ineq-ricci}
|\ric_{\overline{M}}(\overline{V},\overline{W})|\leq \dfrac{m(\kappa-\kappa_0)}{2}
\end{equation}
for any orthogonal pair of vectors $\overline{V},\ \overline{W}\in T\overline{M}$ such that $|\overline{V}|\leq 1,\ |\overline{W}|\leq1.$
In fact, considering the orthonormal referential $\{\overline{e}_1,\overline{e}_2,\ldots,\overline{e}_m,\overline{e}_{m+1}\}$ tangent to $\overline{M},$ we have
\[
\begin{split}
\ric_{\overline{M}}(\overline{V},\overline{V})&=\sum_{i=1}^{m+1}\lan\overline{R}(\overline{V},\overline{e}_i)\overline{V},\overline{e}_i\ran =\sum_{i=1}^{m+1}K_{\overline{M}}(\overline{V},\overline{e}_i)(|\overline{V}|^2-\lan \overline{V},\overline{e}_i\ran^2)\\
&\leq \kappa \sum_{i=1}^{m+1}(|\overline{V}|^2-\lan \overline{V},\overline{e}_i\ran^2) = \kappa [(m+1)|\overline{V}|^2-|\overline{V}|^2]\\
&=m\kappa|\overline{V}|^2.
\end{split}
\]
Analogously, \[\ric_{\overline{M}}(\overline{V},\overline{V})\geq m\kappa_0|\overline{V}|^2.\] Now, since $\ric_{\overline{M}}$ is a symmetric and bilinear form, we have
\[
\begin{split}
\ric_{\overline{M}}(\overline{V},\overline{W})&=\dfrac{1}{4}[\ric_{\overline{M}}(\overline{V}+\overline{W},\overline{V}+\overline{W})-\ric_{\overline{M}}(\overline{V}-\overline{W},\overline{V}-\overline{W})],\\
&\leq\dfrac{1}{4}[m\kappa|\overline{V}+\overline{W}|^2 - m\kappa_0|\overline{V}-\overline{W}|^2]\\
&=\frac{1}{4}[m\kappa(|\overline{V}|^2 + 2\lan\overline{V},\overline{W}\ran + |\overline{W}|^2) - m\kappa_0(|\overline{V}|^2 - 2\lan\overline{V},\overline{W}\ran + |\overline{W}|^2)]\\
&=\dfrac{1}{4}m(\kappa-\kappa_0)(|\overline{V}|^2 + |\overline{W}|^2)\\
&\leq\dfrac{m(\kappa-\kappa_0)}{2}.
\end{split}
\]
Analogously, \[\ric_{\overline{M}}(\overline{V},\overline{W})\geq-\dfrac{m(\kappa-\kappa_0)}{2}.\] Therefore,
\[
-\dfrac{m(\kappa-\kappa_0)}{2}\leq\ric_{\overline{M}}(\overline{V},\overline{W})\leq\dfrac{m(\kappa-\kappa_0)}{2}.
\]
This concludes the proof of the estimate (\ref{ineq-ricci}).
}
\end{remark}

In the next lemma we will estimate $\tr_M\left(E\longmapsto P_1\left(\left(\overline{\n}_E\overline{X}\right)^T\right)\right)$ in terms of $S_1,$ the distance function of $\overline{M}$ an the upper bound $\kappa$ of the sectional curvatures of $\overline{M}.$

\begin{lemma}\label{lemma-main1}
Let $\overline{M}^{m+1},$ $m \geq 2,$ be a Riemannian $(m+1)-$manifold with sectional curvatures bounded from above by a constant $\kappa,$ $M$ be a hypersurface of $\overline{M}^{m+1}$ such that $S_1>0$ and $S_2\geq0,$ and $\rho(x)=\rho(p,x)$ be the geodesic distance of $\overline{M}$ starting at a point $x_0\in \overline{M}.$ If $x\in M$ satisfies $\rho(x)<i(\overline{M}),$ then 

\begin{itemize}
\item[(i)] for $\kappa\leq 0$ and $\overline{X} = \rho\overline{\n}\rho,$
\[
\tr_M\left(E\longmapsto P_1\left(\left(\overline{\n}_E\overline{X}\right)^T\right)\right)(x)\geq (m-1)S_1(x);
\]
\item[(ii)] for $\kappa >0,$ $\rho(x)<\frac{\pi}{2\sqrt{\kappa}}$ and $\overline{X}=\dfrac{1}{\sqrt{\kappa}}(\sin\sqrt{\kappa}\rho)\overline{\n}\rho,$
\[
\tr_M\left(E\longmapsto P_1\left(\left(\overline{\n}_E\overline{X}\right)^T\right)\right)(x)\geq (m-1)S_1(x)(\cos\sqrt{\kappa}\rho(x)).
\]
\end{itemize}
\end{lemma}

\begin{proof}

Let $\gamma:[0,\rho(x)]\ria\overline{M}$ defined by $\gamma(t)=\exp_{x_0}(tu), \ u\in T_{x_0}\overline{M},$ be the unit speed geodesic such that $\gamma(0)=x_0$ e $\gamma(\rho(x))=x.$ Let $\{e_1(x),e_2(x),\ldots,e_m(x)\}$ be an orthonormal basis of $T_xM$ composed by eigenvectors of $P_1$ in $x\in M,$ i.e., 
\[
P_1(e_i(x))=\theta_i(x)e_i(x), \ i=1,2,\ldots,m,
\]
see Remark \ref{P1-selfadjoint}, p.\pageref{P1-selfadjoint}. Let $Y_i,$ $i=1,2,\ldots,m,$ be the unitary projections of $e_i(x)$ over $\gamma'(\rho(x))^\perp\subset T_x\overline{M},$ namely,
\[
Y_i = \dfrac{e_i(x)-\lan e_i(x),\gamma'(\rho(x))\ran\gamma'(\rho(x))}{\|e_i(x)-\lan e_i(x),\gamma'(\rho(x))\ran\gamma'(\rho(x))\|}, \ i=1,2,\ldots,m.
\]
Thus,
\[
e_i(x) = b_i Y_i + c_i \gamma'(\rho(x)),
\]
where $b_i = \|e_i(x)-\lan e_i(x),\gamma'(\rho(x))\ran\gamma'(\rho(x))\|$ and $c_i = \lan e_i(x),\gamma'(\rho(x))\ran$ satisfy $b_i^2+c_i^2=1$ and $Y_i\perp\gamma'$ for all $i=1,2,\ldots,m.$ Since the sectional curvatures of $\overline{M}$ satisfy $K_{\overline{M}}\leq\kappa$ and, in the case we assume $\kappa>0,$ we have $\rho(x)<\frac{\pi}{2\sqrt{\kappa}},$ then do not exist conjugate points to $x_0$ along $\gamma.$ Then
\[
\begin{split}
\tr_M\left(E\longmapsto P_1\left(\left(\overline{\n}_E\overline{X}\right)^T\right)\right) =& \sum_{i=1}^m \lan\overline{\n}_{e_i}\overline{X}, P_1(e_i)\ran = \sum_{i=1}^m\theta_i\lan\overline{\n}_{e_i} \overline{X},e_i\ran\\
=&\sum_{i=1}^m \theta_i\lan\overline{\n}_{b_iY_i+c_i\gamma'}\overline{X},b_iY_i+c_i\gamma'\ran\\
=&\sum_{i=1}^m\theta_ib_i^2\lan\overline{\n}_{Y_i}\overline{X},Y_i\ran + \sum_{i=1}^m\theta_ic^2_i\lan\overline{\n}_{\gamma'}\overline{X},\gamma'\ran\\
&\qquad+\sum_{i=1}^m \theta_ib_ic_i\left[\lan\overline{\n}_{Y_i}\overline{X},\gamma'\ran+\lan\overline{\n}_{\gamma'}\overline{X},Y_i\ran\right].\\
\end{split}
\]
In order to unify the proof, let us denote by
\[
G(\rho)=\left\{
\begin{array}{cc}
\rho,&\ \mbox{if} \ \kappa\leq0;\\
\dfrac{1}{\sqrt{\kappa}}(\sin\sqrt{\kappa}\rho),&\ \mbox{if} \ \kappa>0.\\
\end{array}
\right.
\]
Since $\overline{X}(t)=G(\rho(t))\overline{\n}\rho(t) = G(\rho(t))\gamma'(t)$ and $\overline{\n}_{\gamma'}\gamma'=0,$ we have
\[
\begin{split}
\lan\overline{\n}_{\gamma'}\overline{X},\gamma'\ran &= \lan \overline{\n}_{\gamma'}(G(\rho)\gamma'),\gamma'\ran = \lan G'(\rho)\lan\overline{\n}\rho,\gamma'\ran\gamma' + G(\rho) \overline{\n}_{\gamma'}\gamma',\gamma'\ran\\
&=G'(\rho)\lan\overline{\n}\rho,\gamma'\ran\lan\gamma',\gamma'\ran=G'(\rho),\\
\lan\overline{\n}_{Y_i}\overline{X},\gamma'\ran& = \lan\overline{\n}_{Y_i}(G(\rho)\gamma'),\gamma'\ran = \lan G'(\rho) \lan Y_i,\overline{\n}\rho\ran\gamma' + G(\rho)\overline{\n}_{Y_i}\gamma',\gamma'\ran\\
&=G'(\rho)\lan Y_i,\gamma'\ran + G(\rho)\lan \overline{\n}_{Y_i}\gamma',\gamma'\ran\\
&=\frac{G(\rho)}{2} Y_i\lan\gamma',\gamma'\ran =0,\\
\lan\overline{\n}_{\gamma'}\overline{X},Y_i\ran &= \lan \overline{\n}_{\gamma'}(G(\rho)\gamma'),Y_i\ran= \lan G'(\rho)\lan\gamma',\overline{\n}\rho\ran\gamma' + G(\rho)\overline{\n}_{\gamma'}\gamma', Y_i\ran=0.
\end{split}
\]
Thus,
\[
\tr_M\left(E\longmapsto P_1\left(\left(\overline{\n}_E\overline{X}\right)^T\right)\right) = \sum_{i=1}^m \theta_ib_i^2\lan\overline{\n}_{Y_i}\overline{X},Y_i\ran + \sum_{i=1}^m \theta_ic_i^2.
\]
On the other hand, is well known that
\[
\lan\overline{\n}_{U}\overline{\n}\rho, V\ran = \dfrac{G'(\rho)}{G(\rho)}[\lan U,V\ran - \lan U,\overline{\n}\rho\ran \lan V,\overline{\n}\rho\ran],
\]
for $\R^{m+1}$ and $\s^{m+1}(\kappa).$ Since
\[
\begin{split}
\lan\overline{\n}_{Y_i}\overline{X},Y_i\ran &= \lan\overline{\n}_{Y_i}(G(\rho)\overline{\n}\rho),Y_i\ran\\
& = \lan G'(\rho)\lan Y_i,\overline{\n}\rho\ran \overline{\n}\rho + G(\rho)\overline{\n}_{Y_i}\overline{\n}\rho, Y_i\ran\\
&= G(\rho)\lan\overline{\n}_{Y_i}\overline{\n}\rho, Y_i\ran,
\end{split}
\]
using Hessian comparison theorem for $\overline{M}$ and the model spaces $\R^{m+1}$ for $\kappa\leq 0$ and $\s^{m+1}(\kappa)$ for $\kappa>0,$ we have
\[
\begin{split}
\lan\overline{\n}_{Y_i}\overline{X},Y_i\ran &=G(\rho)\lan\overline{\n}_{Y_i}\overline{\n}\rho, Y_i\ran\\
&\geq G'(\rho)[|Y_i|^2 - \lan Y_i,\overline{\n}\rho\ran^2]\\
&= G'(\rho). 
\end{split}
\]
Therefore,
\[
\begin{split}
\tr_M\left(E\longmapsto P_1\left(\left(\overline{\n}_E\overline{X}\right)^T\right)\right) &= \sum_{i=1}^m \theta_ib_i^2\lan\overline{\n}_{Y_i}\overline{X},Y_i\ran + \sum_{i=1}^m \theta_ic_i^2\\
& = G'(\rho)\sum_{i=1}^m \theta_ib_i^2 + \sum_{i=1}^m \theta_ic_i^2\\
& \geq G'(\rho)\sum_{i=1}^m \theta_i(b_i^2 + c_i^2) = G'(\rho)\sum_{i=1}^m \theta_i\\
& = (m-1)S_1 G'(\rho).
\end{split}
\]
Since $G'(\rho)=1$ for $\kappa\leq 0$ and $G'(\rho) = (\cos\sqrt{\kappa}\rho)\leq 1$ for $\kappa>0,$ we have the result.
\end{proof}


We conclude this section with the following

\begin{lemma}\label{lemma-est-P1}
Let $M$ be a hypersurface of a Riemannian $(m+1)-$manifold $\overline{M}^{m+1}$ such that $S_1>0$ and $S_2\geq0.$ Let $P_1=S_1I-A,$ where $A:TM\ria TM$ denotes the linear operator associated to the second fundamental form of $M$ and $I:TM\ria TM$ denotes the identity operator. Then
\[
|\lan P_1(U),V\ran|\leq 2S_1|U||V|
\]
for all $U,V\in TM.$
\end{lemma}
\begin{proof}
Let $\theta_i,$ $i=1,2,\ldots,m,$ be the eigenvalues of $P_1.$ Since $\theta_i = S_1 - k_i,$ where $k_i$ are the eigenvalues of the second fundamental form $A,$ we have
\begin{equation}\label{est.P1-2}
\begin{split}
\theta_i &= S_1 -k_i \leq S_1 + |k_i|\\
&\leq S_1 + \sqrt{k_1^2 + k_2^2 + \cdots + k_m^2}\\
&= S_1 + |A| = S_1 + \sqrt{S_1^2 - 2S_2}\\
&\leq 2S_1.
\end{split}
\end{equation}
By using that $P_1$ is positive semidefinite, the Cauchy-Schwarz inequality, and the estimate (\ref{est.P1-2}), we obtain
\begin{equation}\label{ineq.est.P1-2}
\begin{split}
|\lan P_1(U),V\ran|&=|\lan\sqrt{P_1}(U),\sqrt{P_1}(V)\ran|\\
                   &\leq|\sqrt{P_1}(U)||\sqrt{P_1}(V)|\\
                   &=\lan P_1(U),U\ran^{1/2}\lan P_1(V),V\ran^{1/2}\\
                   &\leq (2S_1)^{1/2}|U|(2S_1)^{1/2}|V|\\
                   &=2S_1|U||V|.
\end{split} 
\end{equation}
\end{proof}

\section{Proof of the Poincar\'e inequality}

\begin{proof}[Proof of Theorem \ref{Theo-Poincare}.]
{\bf Case $\kappa\leq 0$.} Initially, since $(\diam \Omega)<2i(\overline{M}),$ we can consider $B_r(x_0),$ $x_0\in \overline{M},$  the smallest extrinsic ball containing $\overline{\Omega},$ and $\rho(x)=\rho(x_0,x)$ the extrinsic distance from $x_0$ to $x\in M.$ Since $\Omega\subset B_r(x_0),$ then, for all $x\in\Omega,$ 
\begin{equation}\label{diam-omega}
\rho(x)\leq r = \dfrac{(\diam\Omega)}{2}.
\end{equation}
Multiplying by $f$ the expression (\ref{div-Pr-const}) in the Proposition \ref{prop-div-P1-const}, p. \pageref{div-Pr-const}, for the vector field $\overline{X}=\rho\overline{\n}\rho$ and using item (i) of Lemma \ref{lemma-main1}, p. \pageref{lemma-main1}, we have 
\begin{equation}\label{div-f}
f\di_M(P_1(\rho\n\rho)) \geq (m-1)fS_1 + f\ric_{\overline{M}}(\rho\n\rho,\eta) + 2S_2 f \lan\rho\overline{\n}\rho,\eta\ran.
\end{equation}
This implies
\[
\begin{split}
\di_M(fP_1(\rho\n\rho))&=f\di_M(P_1(\rho\n\rho)) + \lan\n f,P_1(\rho\n\rho)\ran\\
                     &\geq (m-1)fS_1 + f\ric_{\overline{M}}(\rho\n\rho,\eta) + 2S_2 f\lan\rho\overline{\n}\rho,\eta\ran + \lan\n f,P_1(\rho\n\rho)\ran.\\
\end{split}
\]
Integrating expression above over $\Omega,$ using divergence theorem, we have
\[
0\geq (m-1)\int_\Omega fS_1 dM + \int_\Omega f\ric_{\overline{M}}(\rho\n\rho,\eta) dM + 2\int_\Omega S_2 f\lan\rho\overline{\n}\rho,\eta\ran dM + \int_\Omega \lan\n f,P_1(\rho\n\rho)\ran dM,
\]
for every function $f$ compactly supported on $\Omega$, i.e.,
\begin{equation}\label{eq.poincare}
\begin{split}
\int_\Omega fS_1 dM &\leq \frac{1}{m-1}\left[ \int_\Omega \lan\n f,P_1(-\rho\n\rho)\ran dM +  \int_\Omega  f\ric_{\overline{M}}(-\rho\n\rho,\eta) dM\right.\\
&\qquad\qquad\quad\left. + 2\int_\Omega S_2 f\lan-\rho\overline{\n}\rho,\eta\ran dM\right]. 
\end{split}
\end{equation}
By using Lemma \ref{lemma-est-P1} for $U=-\rho\n\rho$ and $V=\n f,$ we have
\[
|\lan \n f, P_1(-\rho\n\rho)\ran| \leq 2S_1\rho|\n f|,
\]
and the estimate (\ref{ineq-ricci}), p. \pageref{ineq-ricci}, gives
\[
\ric_{\overline{M}}(-\n\rho,\eta)\leq \frac{m(\kappa - \kappa_0)}{2}.
\]
Replacing these inequalities in (\ref{eq.poincare}) we obtain
\[
\begin{split}
\int_\Omega fS_1 dM &\leq \frac{2}{m-1}\int_\Omega\rho\left[|\n f|S_1 + \left(\dfrac{m(\kappa-\kappa_0)}{4} + S_2\right)f\right]dM.
\end{split}
\]
Therefore, using (\ref{diam-omega}),
\[
\int_\Omega fS_1 dM \leq \frac{1}{m-1}(\diam \Omega)\int_\Omega\left[|\n f|S_1 + \left(\dfrac{m(\kappa-\kappa_0)}{4} + S_2\right)f\right]dM.
\]

{\bf Case $\kappa>0$.} Again, since $(\diam\Omega)< 2i(\overline{M})$ and $(\diam\Omega)<\frac{\pi}{\sqrt{\kappa}},$ we can consider $B_r(x_0), \ x_0\in\overline{M}^{m+1},$ the smallest extrinsic ball containing $\overline{\Omega},$ and $\rho(x)=\rho(x_0,x)$ the extrinsic distance from $x_0$ to $x\in M.$ Since $\Omega\subset B_r(x_0),$ then 
\begin{equation}\label{diam-omega-2}
\rho(x)\leq r = \dfrac{(\diam\Omega)}{2},
\end{equation}
for all $x\in\Omega.$ 
By using the Proposition \ref{prop-div-P1-const} for $\overline{X}=\frac{1}{\sqrt{\kappa}}(\sin\sqrt{\kappa}\rho)\overline{\n}\rho$ and using the Lemma \ref{lemma-main1}, item (ii), we have
\[
\begin{split}
f\di_M\left(P_1\left(\frac{(\sin\sqrt{\kappa}\rho)}{\sqrt{\kappa}}\n\rho\right)\right) &\geq (m-1)fS_1(\cos\sqrt{\kappa}\rho) + \frac{(\sin\sqrt{\kappa}\rho)}{\sqrt{\kappa}}f\ric_{\overline{M}}(\n\rho,\eta)\\
&\qquad +2S_2f\frac{(\sin\sqrt{\kappa}\rho)}{\sqrt{\kappa}}\lan\overline{\n}\rho,\eta\ran.\\
\end{split}
\]
This implies that
\[
\begin{split}
\di_M(fP_1((\sin\sqrt{\kappa}\rho)\n\rho))&=f\di_M(P_1((\sin\sqrt{\kappa}\rho)\n\rho)) + (\sin\sqrt{\kappa}\rho)\lan\n f,P_1(\n\rho)\ran\\
&\geq (m-1)fS_1\sqrt{\kappa}(\cos\sqrt{\kappa}\rho) + (\sin\sqrt{\kappa}\rho)f\ric_{\overline{M}}(\n\rho,\eta)\\
&\qquad+2S_2f(\sin\sqrt{\kappa}\rho)\lan\overline{\n}\rho,\eta\ran + (\sin\sqrt{\kappa}\rho)\lan\n f,P_1(\n\rho)\ran.\\
\end{split}
\]
Integrating the expression above over $\Omega$ and applying divergence theorem, we obtain
\begin{equation}\label{eq.div-sph}
\begin{split}
0&\geq (m-1)\sqrt{\kappa}\int_\Omega fS_1(\cos\sqrt{\kappa}\rho)dM + \int_\Omega(\sin\sqrt{\kappa}\rho)f\ric_{\overline{M}}(\n\rho,\eta)dM\\
&\qquad+\int_\Omega 2S_2f(\sin\sqrt{\kappa}\rho)\lan\overline{\n}\rho,\eta\ran dM + \int_\Omega(\sin\sqrt{\kappa}\rho)\lan\n f,P_1(\n\rho)\ran dM,\\
\end{split}
\end{equation}
since $f$ is compactly supported in $\Omega.$ By using the Lemma \ref{lemma-est-P1}, p. \pageref{lemma-est-P1}, for $U=-\n\rho$ and $V=\n f,$ we have
\[
|\lan \n f, P_1(-\n\rho)\ran| \leq 2S_1|\n f|.
\]
Replacing this estimate in (\ref{eq.div-sph}), and using the estimate (\ref{ineq-ricci}), p. \pageref{ineq-ricci}, we obtain
\[
\begin{split}
\int_\Omega fS_1(\cos\sqrt{\kappa}\rho)dM&\leq \frac{1}{(m-1)\sqrt{\kappa}}\int_\Omega(\sin\sqrt{\kappa}\rho)\left[\lan\n f, P_1(\n \rho)\ran + f\ric_{\overline{M}}(\n\rho,\eta) + 2S_2f\right]dM\\
&\leq \frac{2}{(m-1)\sqrt{\kappa}}\int_\Omega(\sin\sqrt{\kappa}\rho)\left[|\n f|S_1 + \left(\frac{m(\kappa-\kappa_0)}{4} + S_2\right)f\right]dM.
\end{split}
\]
Since $\rho(x)\leq r<\frac{\pi}{2\sqrt{\kappa}}$ for all $x\in\Omega,$ $(\cos\sqrt{\kappa}\rho)$ is a decreasing function and $(\sin\sqrt{\kappa}\rho)$ is an increasing function for $\rho\in\left(0,\frac{\pi}{2\sqrt{\kappa}}\right),$ we have
\[
\begin{split}
(\cos\sqrt{\kappa}r)\int_\Omega fS_1 dM &\leq \int_\Omega fS_1(\cos\sqrt{\kappa}\rho)dM\\
                                        &\leq\frac{2}{(m-1)\sqrt{\kappa}}\int_\Omega(\sin\sqrt{\kappa}\rho)\left[|\n f|S_1 + \left(\frac{m(\kappa-\kappa_0)}{4} + S_2\right)f\right]dM\\
                                        &\leq\frac{2}{(m-1)\sqrt{\kappa}}(\sin\sqrt{\kappa}r)\int_\Omega\left[|\n f|S_1 + \left(\frac{m(\kappa-\kappa_0)}{4} + S_2\right)f\right]dM.\\
\end{split}
\]
i.e.,
\[
\int_\Omega fS_1 dM \leq \frac{2(\tan\sqrt{\kappa}r)}{\sqrt{\kappa}(m-1)}\int_\Omega\left[|\n f|S_1 + \left(\frac{m(\kappa-\kappa_0)}{4} + S_2\right)f\right]dM.
\]
Therefore, using (\ref{diam-omega-2}),
\[
\int_\Omega fS_1 dM \leq \frac{2\tan\left(\frac{\sqrt{\kappa}}{2}(\diam\Omega)\right)}{\sqrt{\kappa}(m-1)}\int_\Omega\left[|\n f|S_1 + \left(\frac{m(\kappa-\kappa_0)}{4} + S_2\right)f\right]dM,
\] 
for $(\diam\Omega)<\dfrac{\pi}{\sqrt{\kappa}}.$
\end{proof}

\section{Proof of the mean value inequalities and its applications}

From now on, we will let $B_r=B_r(x_0)$ be the extrinsic open ball with center $x_0\in\overline{M}^{m+1}$ and radius $r.$ If $\partial M\neq \emptyset,$ assume in addition that $B_r(x_0)\cap\partial M=\emptyset.$ 

\begin{proof}[Proof of Theorem \ref{Theo-Mean}.]
Let $\overline{X}$ be a vector field on the ambient space. Since 
\[
\di_M(P_1(fX^T)) = \lan \n f, P_1(X^T)\ran + f\di_M(P_1(X^T)),
\]
using the Proposition \ref{prop-div-P1-const}, p. \pageref{prop-div-P1-const}, we have
\[
\di_M(P_1(fX^T)) = \lan \n f,P_1(X^T)\ran + f\tr_M\left(E\mapsto P_1((\overline{\n}_E\overline{X})^T)\right) + f\ric_{\overline{M}}(X^T,\eta) + 2S_2f\lan \overline{X},\eta\ran.
\]
By using integration, the divergence theorem, Lemma \ref{lemma-est-P1}, and the co-area formula, we have
\[
\begin{split}
\int_{M\cap B_r} \di_M(P_1(fX^T))dM&=\int_{\partial(M\cap B_r)}\lan P_1(fX^T),\nu\ran dS_M\\
&\leq 2\int_{\partial(M\cap B_r)} fS_1|X^T|dS_M\\
&\leq 2\sup_{\partial(M\cap B_r)}|X^T|\int_{\partial(M\cap B_r)} fS_1|\n\rho|^{-1}dS_M\\
&=2\sup_{\partial(M\cap B_r)}|X^T|\dfrac{d}{dr}\left(\int_{M\cap B_r} fS_1 dM\right),\\
\end{split}
\]
where $\nu$ is the outer conormal vector field of $\partial(M\cap B_r)$ and $dS_M$ is the volume element of $\partial(M\cap B_r).$ This implies
\[
\begin{split}
2\sup_{\partial(M\cap B_r)}|X^T|&\dfrac{d}{dr}\left(\int_{M\cap B_r}fS_1 dM\right)\geq \int_{M\cap B_r}\lan \n f, P_1(X^T)\ran dM\\
& + \int_{M\cap B_r} f\tr_M\left(E\mapsto P_1((\overline{\n}_E\overline{X})^T)\right)dM +\int_{M\cap B_r}f\ric_{\overline{M}}(X^T,\eta)dM\\
&\qquad\qquad\qquad + 2\int_{M\cap B_r}S_2f\lan\overline{X},\eta\ran dM.\\
\end{split}
\]
{\bf Case $\kappa\leq 0$.} Taking $\overline{X}=\rho\overline{\n}\rho,$ we have $\sup_{\partial(M\cap B_r)}|X^T|=r$ and, using Lemma \ref{lemma-main1}, p. \pageref{lemma-main1}, we have
\[
\begin{split}
2r\dfrac{d}{dr}\left(\int_{M\cap B_r}fS_1 dM\right)&\geq \int_{M\cap B_r}\lan \n f, P_1(\rho\n\rho)\ran dM+ (m-1)\int_{M\cap B_r} fS_1dM\\
& \qquad +\int_{M\cap B_r}f\ric_{\overline{M}}(\rho\n\rho,\eta)dM+ 2\int_{M\cap B_r}S_2f\lan\rho\overline{\n}\rho,\eta\ran dM,\\
\end{split}
\]
i.e.,
\[
\begin{split}
r\dfrac{d}{dr}\left(\int_{M\cap B_r} fS_1 dM \right) &- \dfrac{m-1}{2}\int_{M\cap B_r} fS_1 dM\\
&\geq \dfrac{1}{2}\int_{M\cap B_r}\left[\lan\rho\overline{\n}\rho, P_1(\n f) + 2S_2f\eta\ran + f\ric_{\overline{M}}(\rho\n\rho,\eta)\right] dM.
\end{split}
\]
Since
\[
r\dfrac{d}{dr}\left(\int_{M\cap B_r} fS_1 dM \right) - \dfrac{m-1}{2}\int_{M\cap B_r} fS_1 dM = r^{\frac{m+1}{2}}\dfrac{d}{dr}\left(\frac{1}{r^{\frac{m-1}{2}}}\int_{M\cap B_r}fS_1 dM\right),
\]
we have, dividing by $r^{\frac{m+1}{2}},$
\begin{equation}\label{eq.diff-2}
\dfrac{d}{dr}\left(\frac{1}{r^{\frac{m-1}{2}}}\int_{M\cap B_r}fS_1 dM\right)\geq \dfrac{1}{2r^{\frac{m+1}{2}}}\int_{M\cap B_r}\left[\lan\rho\overline{\n}\rho, P_1(\n f) + 2S_2f\eta\ran + f\ric_{\overline{M}}(\rho\n\rho,\eta)\right] dM.
\end{equation}
Integrating the inequality (\ref{eq.diff-2}) from $s$ to $t,$ we have the result for the case $\kappa\leq 0.$ 

{\bf Case $\kappa>0$.} Consider $\overline{X}=\frac{1}{\sqrt{\kappa}}(\sin\sqrt{\kappa}\rho)\overline{\n}\rho.$ Since we are assuming $\rho<\frac{\pi}{2\sqrt{\kappa}},$ we have 
\[
\sup_{\partial(M\cap B_r)}|X^T| = \dfrac{1}{\sqrt{\kappa}}(\sin \sqrt{\kappa}r)
\]
and, using Lemma \ref{lemma-main1}, we obtain
\[
\begin{split}
\frac{2}{\sqrt{\kappa}}(\sin\sqrt{\kappa}r)&\dfrac{d}{dr}\left(\int_{M\cap B_r} fS_1 dM\right)\geq \int_{M\cap B_r} \dfrac{1}{\sqrt{\kappa}}(\sin\sqrt{\kappa}\rho)\lan\n \rho, P_1(\n f)\ran dM\\
& +(m-1)\int_{M\cap B_r}(\cos\sqrt{\kappa}\rho)fS_1 dM +\int_{M\cap B_r}\frac{(\sin\sqrt{\kappa}\rho)}{\sqrt{\kappa}} f\ric_{\overline{M}}(\n\rho,\eta) dM\\
&\qquad + 2\int_{M\cap B_r} \frac{(\sin\sqrt{\kappa}\rho)}{\sqrt{\kappa}}S_2f\lan \overline{\n}\rho,\eta\ran dM.
\end{split}
\]
Since $(\cos\sqrt{\kappa}\rho)$ is a decreasing function for $\rho\leq\frac{\pi}{2\sqrt{\kappa}},$ we have
\[
\int_{M\cap B_r}(\cos\sqrt{\kappa}\rho) fS_1 dM \geq (\cos\sqrt{\kappa}r)\int_{M\cap B_r} fS_1 dM,
\]
and thus, dividing by $\frac{2}{\sqrt{\kappa}}(\sin\sqrt{\kappa}r)$ we obtain
\[
\begin{split}
\dfrac{d}{dr}\left(\int_{M\cap B_r} fS_1 dM\right) &- \dfrac{m-1}{2}\sqrt{\kappa}(\cot\sqrt{\kappa}r)\int_{M\cap B_r} fS_1 dM\\
&\!\!\!\!\!\!\!\!\!\!\!\!\!\!\!\!\!\!\!\!\!\!\!\!\!\!\!\!\!
\geq \dfrac{1}{2(\sin\sqrt{\kappa}r)}\int_{M\cap B_r}(\sin\sqrt{\kappa}\rho)\left[\lan\overline{\n}\rho, P_1(\n f) + 2S_2f\eta\ran+ f\ric_{\overline{M}}(\n\rho,\eta)\right] dM.\\
\end{split}
\]
Since
\[
\begin{split}
\dfrac{d}{dr}\left(\int_{M\cap B_r} fS_1 dM\right) &- \frac{m-1}{2}\sqrt{\kappa}(\cot\sqrt{\kappa}r)\int_{M\cap B_r} fS_1 dM\\
& = (\sin \sqrt{\kappa}r)^{\frac{m-1}{2}}\dfrac{d}{dr}\left(\frac{1}{(\sin\sqrt{\kappa} r)^{\frac{m-1}{2}}}\int_{M\cap B_r} fS_1 dM\right),\\
\end{split}
\]
we have
\begin{equation}\label{eq.diff-1}
\begin{split}
\dfrac{d}{dr}&\left(\frac{1}{(\sin\sqrt{\kappa}r)^{\frac{m-1}{2}}}\int_{M\cap B_r} fS_1 dM\right)\\
& \geq \frac{1}{2(\sin\sqrt{\kappa}r)^{\frac{m+1}{2}}}\int_{M\cap B_r}(\sin\sqrt{\kappa}\rho)\left[\lan\overline{\n}\rho, P_1(\n f) + 2S_2f\eta\ran+ f\ric_{\overline{M}}(\n\rho,\eta)\right] dM.\\
\end{split}
\end{equation}
Integrating the inequality (\ref{eq.diff-1}) from $s$ to $t,$ we have the result for the case $\kappa>0.$
\end{proof}

\begin{remark}\label{convex case}{\normalfont\ If $A\geq0,$ then we can estimate the eigenvalues of $P_1$ by $\theta_i=S_1-k_i\leq S_1$ in the place of $\theta_i\leq 2S_1$ in the proof of the Lemma \ref{lemma-est-P1}, p. \pageref{lemma-est-P1}. In this case, the exponents of the mean value inequalities become $(m-1)$ in the place of $\frac{m-1}{2}$ and the mean value inequalities become:
\item[(i)] for $\kappa\leq 0,$ 
\[
\begin{split}
\dfrac{1}{t^{m-1}}& \int_{M\cap B_t}fS_1dM - \dfrac{1}{s^{m-1}}\int_{M\cap B_s}fS_1dM\\
&\geq \int_s^t \frac{1}{r^{m}}\int_{M\cap B_r}[\left\lan\rho\overline{\n}\rho,P_1(\n f) + 2S_2f\eta\right\ran + f\ric_{\overline{M}}(\rho\n\rho,\eta)] dM dr;
\end{split}
\]
\item[(ii)] for $\kappa>0$ and $t<\dfrac{\pi}{2\sqrt{\kappa}},$
\[
\begin{split}
&\dfrac{1}{(\sin\sqrt{\kappa} t)^{m-1}} \int_{M\cap B_t}fS_1dM - \dfrac{1}{(\sin\sqrt{\kappa} s)^{m-1}}\int_{M\cap B_s}fS_1dM\\
&\qquad \geq \int_s^t \frac{1}{(\sin\sqrt{\kappa}r)^{m}}\int_{M\cap B_r}(\sin\sqrt{\kappa}\rho) [\lan\overline{\n}\rho,P_1(\n f) + 2S_2f\eta\ran + f\ric_{\overline{M}}(\n\rho,\eta)] dM dr.
\end{split}
\]
}
\end{remark}

Now will prove the corollaries stated in the introduction of this paper.

\begin{proof}[Proof of Corollary \ref{monotonicity}.] Let us prove the case $\kappa>0.$ The case $\kappa\leq0$ is entirely analogous. Applying the mean value inequality of Theorem \ref{Theo-Mean}, p. \pageref{Theo-Mean}, item (ii), for $f\equiv1,$ we have
\[
\begin{split}
\dfrac{1}{(\sin\sqrt{\kappa} t)^{\frac{m-1}{2}}}\int_{M\cap B_t} S_1 dM& - \dfrac{1}{(\sin\sqrt{\kappa} s)^{\frac{m-1}{2}}}\int_{M\cap B_s} S_1 dM \\
&\!\!\!\!\!\!\!\!\!\geq \int_s^t \frac{1}{(\sin\sqrt{\kappa}r)^{\frac{m+1}{2}}}\int_{M\cap B_r}(\sin\sqrt{\kappa}\rho) [2S_2\lan\overline{\n}\rho,\eta\ran + \ric_{\overline{M}}(\n\rho,\eta)] dM dr.\\
\end{split}
\]
Since
\[
\begin{split}
\int_{M\cap B_r}\!\!\!(\sin\sqrt{\kappa}\rho) \left[2S_2\lan\overline{\n}\rho,\eta\ran + \ric_{\overline{M}}(\n\rho,\eta)\right] dM \geq - 2(\sin\sqrt{\kappa}r)\int_{M\cap B_r}\left(\dfrac{m(\kappa-\kappa_0)}{4} + S_2\right)dM,\\
\end{split}
\]
and using the hypothesis (\ref{hyp-mono-geral}), p. \pageref{hyp-mono-geral}, we have, for $0<s<t<R_0,$
\[
\begin{split}
\dfrac{1}{(\sin\sqrt{\kappa} t)^{\frac{m-1}{2}}}\int_{M\cap B_t}S_1dM &- \dfrac{1}{(\sin\sqrt{\kappa} s)^{\frac{m-1}{2}}}\int_{M\cap B_s}S_1dM\\
&\geq - \int_s^t(\sin\sqrt{\kappa}r)^{-\frac{m-1}{2}}\int_{M\cap B_r}\left(\dfrac{m(\kappa-\kappa_0)}{4}+S_2\right)dM dr\\
&\geq-\al\Lambda \int_s^t(\sin\sqrt{\kappa}r)^{-\frac{m-1}{2}}\left(\dfrac{r}{R_0}\right)^{\al-1}\int_{M\cap B_r}S_1dM dr.\\
\end{split}
\]
Letting $\displaystyle{g(r)=\dfrac{1}{(\sin\sqrt{\kappa}r)^{\frac{m-1}{2}}}\int_{M\cap B_r}S_1dM}$ and dividing by $t-s,$ inequality above becomes
\[
\begin{split}
\dfrac{g(t)-g(s)}{t-s}&\geq-\dfrac{1}{t-s}\int_s^t \al\Lambda\left(\dfrac{r}{R_0}\right)^{\al-1}g(r)dr\\
&=-\dfrac{1}{t-s}\left[\int_\ve^t \al\Lambda\left(\dfrac{r}{R_0}\right)^{\al-1}g(r)dr - \int_\ve^s \al\Lambda\left(\dfrac{r}{R_0}\right)^{\al-1}g(r)dr\right],
\end{split}
\]
for every $\ve>0$ sufficiently small.

Since $\displaystyle{r\mapsto\int_{M\cap B_r} S_1dM}$ is a monotone non-decreasing function, a classical theorem of Lebesgue guarantee that this function is differentiable almost everywhere with respect to Lebesgue measure of $\R$.
Consequently, the same holds for $g(r).$ Considering the points $s$ such that $g$ is differentiable and taking $t\ria s,$ the left hand side goes to $g'(s)$ and the right hand side goes to $\al\Lambda\left(\dfrac{s}{R_0}\right)^{\al-1}g(s).$ Thus, $g$ satisfies
\[
g'(r) + \al\Lambda\left(\frac{r}{R_0}\right)^{\alpha-1}g(r)\geq 0.
\]  
Since
\[
\dfrac{d}{dr}\left(\exp(\Lambda R_0^{1-\al}r^\al)g(r)\right)=\exp(\Lambda R_0^{1-\al}r^\al)\left(\al\Lambda\left(\dfrac{r}{R_0}\right)^{\al-1}g(r) + g'(r)\right)\geq0,
\]
we conclude that $h(r)=\exp(\Lambda R_0^{1-\al}r^\al)g(r)$ is monotone non-decreasing for every $r\in(0,R_0).$
\end{proof}

\begin{proof}[Proof of Corollary \ref{Lp}.] Again, we prove the case $\kappa>0.$ The case $\kappa\leq0$ is entirely analogous. 
Applying inequality (\ref{eq.diff-1}) to $f\equiv 1$ and using that $(\sin\sqrt{\kappa}\rho)$ is a increasing function, we have
\[
\begin{split}
\dfrac{d}{dr}&\left(\frac{1}{(\sin\sqrt{\kappa}r)^{\frac{m-1}{2}}}\int_{M\cap B_r} S_1 dM\right)\\
& \geq \frac{1}{2(\sin\sqrt{\kappa}r)^{\frac{m+1}{2}}}\int_{M\cap B_r}(\sin\sqrt{\kappa}\rho)\left[2S_2\lan\overline{\n}\rho,\eta\ran+ \ric_{\overline{M}}(\n\rho,\eta)\right] dM.\\
&\geq - \frac{1}{(\sin\sqrt{\kappa}r)^{\frac{m-1}{2}}}\int_{M\cap B_r}\left(\dfrac{m(\kappa-\kappa_0)}{4} + S_2\right) dM. \\
\end{split}
\]
By using H\"older inequality and the hypothesis, we obtain
\[
\begin{split}
\int_{M\cap B_r}\left(\dfrac{m(\kappa-\kappa_0)}{4} + S_2\right)& dM \leq \left(\int_{M\cap B_r}\left(\dfrac{m(\kappa-\kappa_0)}{4} + S_2\right)^p dM\right)^{\frac{1}{p}}\left(\int_{M\cap B_r}dM\right)^{1-\frac{1}{p}}\\
&\leq \dfrac{1}{c^{1-\frac{1}{p}}}\left(\int_{M\cap B_{R_0}}\left(\dfrac{m(\kappa-\kappa_0)}{4} + S_2\right)^p dM\right)^{\frac{1}{p}}\left(\int_{M\cap B_r}S_1 dM\right)^{1-\frac{1}{p}} \\
&\leq \frac{\Lambda}{c^{1-\frac{1}{p}}}\left(\int_{M\cap B_r}S_1 dM\right)^{1-\frac{1}{p}} dM.\\
\end{split}
\]
This implies
\[
\dfrac{d}{dr}\left(\frac{1}{(\sin\sqrt{\kappa}r)^{\frac{m-1}{2}}}\int_{M\cap B_r} S_1dM\right)\geq -\frac{\Lambda}{c^{1-1/p}}\frac{1}{(\sin\sqrt{\kappa}r)^{\frac{m-1}{2}}}\left(\int_{M\cap B_r} S_1dM\right)^{1-\frac{1}{p}}.
\]
Thus,
\[
\begin{split}
\dfrac{d}{dr}\left(\left(\frac{1}{(\sin\sqrt{\kappa}r)^{\frac{m-1}{2}}}\int_{M\cap B_r} S_1dM\right)^{\frac{1}{p}}\right)&=\frac{1}{p}\left(\frac{1}{(\sin\sqrt{\kappa}r)^{\frac{m-1}{2}}}\int_{M\cap B_r} S_1dM\right)^{\frac{1}{p}-1}\\
&\qquad\times \dfrac{d}{dr}\left(\frac{1}{(\sin\sqrt{\kappa}r)^{\frac{m-1}{2}}}\int_{M\cap B_r} S_1dM\right)\\
&\geq-\frac{1}{p}\left(\frac{1}{(\sin\sqrt{\kappa}r)^{\frac{m-1}{2}}}\int_{M\cap B_r} S_1dM\right)^{\frac{1}{p}-1}\\
&\qquad\times\frac{\Lambda}{c^{1-1/p}}\frac{1}{(\sin\sqrt{\kappa}r)^{\frac{m-1}{2}}}\left(\!\int_{M\cap B_r} S_1dM\right)^{1-\frac{1}{p}}\\
&=-\frac{\Lambda}{pc^{1-1/p}}\frac{1}{(\sin\sqrt{\kappa}r)^{\frac{m-1}{2p}}}.\\
\end{split}
\]
Integrating inequality above from $s$ to $t,$ we have
\[
\begin{split}
\left(\frac{1}{(\sin\sqrt{\kappa}t)^{\frac{m-1}{2}}}\int_{M\cap B_t} S_1dM\right)^{\frac{1}{p}} &- \left(\frac{1}{(\sin\sqrt{\kappa}s)^{\frac{m-1}{2}}}\int_{M\cap B_s} S_1dM\right)^{\frac{1}{p}}\\
&\geq-\frac{\Lambda}{pc^{1-1/p}}\int_s^t\frac{1}{(\sin\sqrt{\kappa}r)^{\frac{m-1}{2p}}}dr.
\end{split}
\]
\end{proof}

We conclude this paper proving Corollary \ref{teo-self}, p. \pageref{teo-self}.

\begin{proof}[Proof of Corollary \ref{teo-self}.]
Taking $f\equiv1$ in the inequality (\ref{eq.diff-2}), p. \pageref{eq.diff-2}, for $\overline{M}=\R^{m+1},$ using that $M$ is a self-shrinker and the hypothesis $0\leq R\leq\Lambda,$ we have
\begin{equation}\label{self-proof}
\begin{split}
\frac{d}{dr}\left(\frac{1}{r^{\frac{m-1}{2}}}\int_{M\cap B_r}HdM\right) &\geq \frac{m-1}{r^{\frac{m+1}{2}}}\int_{M\cap B_r} R\left(\frac{1}{2}\lan\rho\overline{\n}\rho, \eta\ran\right)dM\\
&=-\frac{m(m-1)}{r^{\frac{m+1}{2}}}\int_{M\cap B_r} RH dM\\
&\geq-\frac{\Lambda m(m-1)}{r^{\frac{m+1}{2}}}\int_{M\cap B_r}HdM.
\end{split}
\end{equation}
Denoting by $\displaystyle{f(r)=\frac{1}{r^{\frac{m-1}{2}}}\int_{M\cap B_r}HdM},$ the inequality (\ref{self-proof}) becomes
\[
f'(r)\geq -\frac{\Lambda m(m-1)}{r}f(r),
\]
which is equivalent to 
\[
\frac{d}{dr}\left(\ln\left(r^{\Lambda m(m-1)}f(r)\right)\right)\geq 0.
\]
This implies that $\ln\left(r^{\Lambda m(m-1)}f(r)\right).$ Therefore $r^{\Lambda m(m-1)}f(r)$ is monotone non-decreasing, i.e.,
\[
\varphi(r)=\frac{1}{r^{(m-1)\left(\frac{1}{2}-\Lambda m\right)}}\int_{M\cap B_r} H dM
\]
is monotone non-decreasing. The monotonicity of the function $\varphi$ implies
\[
\int_{M\cap B_r} H dM \geq r^{(m-1)(\frac{1}{2} - \Lambda m)}\frac{1}{r_0^{(m-1)(\frac{1}{2} - \Lambda m)}}\int_{M\cap r_0} H dM
\]
for $r\geq r_0.$ Therefore, in the case that $0\leq R\leq\Lambda<\frac{1}{2m},$ taking $r\ria\infty$ we obtain $\displaystyle{\int_ M H dM =\infty.}$
\end{proof}

\end{document}